\documentclass{amsart}
\usepackage{amssymb}
\usepackage{amssymb,amsbsy}
\usepackage{epstopdf}
\usepackage{float}
\usepackage{amsfonts}
\usepackage{amsmath}
\usepackage{amscd}
\usepackage{tikz-cd}
\usepackage{latexsym}
\usepackage{bbm}
\usepackage{verbatim}
\usepackage{paralist}
\usepackage[utf8]{inputenc}
\usepackage{etoolbox}
\usepackage{enumitem}
\usepackage{xspace}
\usetikzlibrary{patterns}
\usepackage{hyperref}

\newtheorem{theorem}{Theorem}[section]
\newtheorem{corollary}[theorem]{Corollary}
\newtheorem{lemma}[theorem]{Lemma}
\newtheorem{observation}[theorem]{Observation}

\newtheorem{fact}[theorem]{Fact}
\newtheorem*{theorem*}{Theorem} 

\theoremstyle{definition}
\newtheorem{definition}[theorem]{Definition}

\theoremstyle{remark}
\newtheorem{remark}[theorem]{Remark}

\newtheorem*{claim}{Main Claim}
\newtheorem{claimno}{Claim}

\setlist[description]{leftmargin=5.5em,labelindent=\parindent}

\renewcommand{\P}{\mathbb{P}}

\newcommand{\M}{\mathcal{M}}
\newcommand{\D}{\mathbb{D}}
\newcommand{\N}{{\overline{N}}}

\newcommand{\G}{\overline{G}}
\renewcommand{\S}{{\overline{S}}}

\newcommand{\U}{\mathcal{U}}

\newcommand{\ZFC}{\textup{\ensuremath{\textsf{ZFC}}}}

\newcommand{\CH}{\textup{\textsf{CH}}}

\newcommand{\Ord}{\textup{\ensuremath{\text{Ord}}}}
\newcommand{\id}{\textup{\ensuremath{\text{id}}}}

\DeclareMathOperator{\height}{height}
\DeclareMathOperator{\ran}{range}

\DeclareMathOperator{\cp}{cp}
\DeclareMathOperator{\dom}{dom}

\DeclareMathOperator{\wfc}{wfc}

\DeclareMathOperator*{\bigdoublevee}{\bigvee\mkern-15mu\bigvee}

\newcommand{\To}{\longrightarrow}
\newcommand{\st}{\; | \;}
\newcommand{\set}[2]{\left\{#1\st #2 \right\}}
\newcommand{\seq}[2]{\langle #1 \st #2 \rangle}

\newcommand{\forces}{\Vdash}

\newcommand{\rest}{\mathbin{\upharpoonright}}

\newcommand{\SH}{\mathcal{H}\textit{ull} \,}
\newcommand{\sk}[3]{\SH^{#1}( {#2} \cup {\ran(#3)} ) }
\newcommand{\Sk}[3]{\SH^{#1}( {#2} \cup {#3} ) }

\begin{document}
\title{The subcompleteness of diagonal Prikry forcing}
\author{Kaethe Minden}
 \address[K.~Minden]{Mathematics, Marlboro College, 2582 South Road, Marlboro, VT 05344}
\email{kminden@marlboro.edu}
\urladdr{https://kaetheminden.wordpress.com/}
\date{}     					
\thanks{The material presented here is based on the author's doctoral thesis \cite{Minden:2017fr}, written at the Graduate Center of CUNY under the supervision of Gunter Fuchs.}


\begin{abstract}
Let \(D\) be an infinite discrete set of measurable cardinals. It is shown that generalized Prikry forcing to add a countable sequence to each cardinal in \(D\) is subcomplete. To do this it is shown that a simplified version of generalized Prikry forcing which adds a point below each cardinal in \(D\), called generalized diagonal Prikry forcing, is subcomplete. Moreover, the generalized diagonal Prikry forcing associated to \(D\) is subcomplete above $\mu$, where $\mu$ is any regular cardinal below the first limit point of \(D\).
\end{abstract}
\maketitle
\section{Introduction}
Subcomplete forcing notions are a family of forcing notions that do not add reals and may be iterated using revised countable support. Examples of subcomplete forcing include all countably closed forcing, Prikry forcing, and Namba forcing (under $\CH$), as shown by Ronald B.~Jensen~\cite[Section 3.3]{Jensen:2014}. It is clear from these examples that not all subcomplete forcing notions are proper, and, conversely, nontrivial $ccc$ forcing notions are never subcomplete \cite{Minden:2017fr}. For more on subcomplete forcing and its characteristics, refer to Jensen (primarily \cite{Jensen:2014} and \cite{Jensen:2009fe}), and the author's doctoral thesis \cite{Minden:2017fr}.

Recently, Fuchs~\cite{Fuchs:2017Magidor} has shown that Magidor forcing is also subcomplete. Here an adaptation of Jensen's proof showing that Prikry forcing is subcomplete is employed to see that many Prikry-like forcing notions, in particular those which will be referred to here as generalized diagonal Prikry forcing, following Fuchs~\cite{Fuchs:2005kx}, are subcomplete.

In section~\ref{sec:preliminaries}, some preliminary topics are introduced that may be found in Jensen's lecture notes from the 2012 AII Summer School in Singapore (for the published version refer to \cite{Jensen:2014}, here the version posted on Jensen's website is what is referenced). The notion of the weight of a forcing and fullness of a structure are then introduced, which lead to the definition of subcompleteness given in section~\ref{subsec:subcomplete} and subcompleteness above $\mu$ in section~\ref{subsec:subcompleteabovemu}. Important techniques of Barwise theory are introduced in section~\ref{subsec:BarwiseTheory}, and the ultrapower-like notion of a liftup, to obtain full structures, is introduced in section~\ref{subsec:Liftups}. In section~\ref{sec:GenDiagonalPrikryForcing} the definition of generalized diagonal Prikry foring is given, and it is shown that such forcing notions are subcomplete. Moreover, it is shown that diagonal Prikry forcing is subcomplete above $\mu$ where $\mu$ is a regular cardinal below the first limit point of the infinite discrete set of measurable cardinals used in the forcing.

\section{Preliminaries} \label{sec:preliminaries}
Before defining subcompleteness and generalized diagonal Prikry forcing, some preliminary information is necessary.

Forcing notions $\P = \langle \P, \leq \rangle$ are taken to be partially ordered sets that are separative and contain a ``top" element weaker than all elements of $\P$, denoted $\mathbbm 1$.
Use $N$, M, $\N$ (potentially with subscripts) to always denote transitive models of $\ZFC^-$, the axioms of Zermelo-Fraenkel Set Theory without the axiom of \textsf{Powerset}, and with the axiom of \textsf{Collection}\footnote{ \textsf{Collection} is the following schema: $\forall \vec{w}\, [\forall x \, \exists y \; \varphi(x, y, \vec{w}) \implies \forall A \exists B \, \forall x \in A \,\exists y \in B \; \varphi(x, y, \vec w) ]$.} instead of \textsf{Replacement}. For $\sigma$ an elementary embedding from $M$ to $N$, write $\sigma: M \prec N$ and in case $\sigma$ is the identity, write $M \preccurlyeq N$ to denote and emphasize that $M$ is an elementary substructure of $N$. 
If a map $\sigma$ satisfies $\sigma(\overline a)=a$ and $\sigma(\overline b)=b$, this will be abbreviated as $\sigma(\overline a,\overline b)=a,b$.
Write $\height(N)$ to mean $\Ord \cap N$ and $\alpha^N$ for $\alpha \cap N$ for ordinal $\alpha$.
If $A$ is a set of ordinals, then $\lim(A)$ is the set of limit points in $A$.  
For $\theta$ a cardinal, the collection of sets hereditarily of size less than $\theta$ will be referred to as $H_\theta$. Relativizing the concept to a particular model of set theory, $M$, write $H_\theta^M$ to mean the collection of sets in $M$ that are hereditarily of size less than $\theta$ in $M$. In this case, if $\theta$ is determined by some computation, the computation is meant to take place in $M$.
For a cardinal $\tau$, with abuse of notation, write $L_\tau[A]$ to refer to the structure $\langle L_\tau[A]; \in, A \cap L_\tau[A] \rangle$.

\subsection{The weight of a forcing notion}
\label{subsec:delta}
Jensen defines what shall be called the weight of a Boolean algebra \cite[Section 3.1 p.~31]{Jensen:2014}, but the concept is equally applicable to any poset.

\begin{definition} For a poset $\P$, the \emph{weight} of \(\P\) is the least cardinality $\delta(\P)$ of a dense subset of $\P$. 
\end{definition}

As Jensen states, for forcing notions $\P$, the weight can be replaced with the cardinality of $\P$, or even $\P$, for the purpose of defining subcompleteness. However, $\delta(\P)$ and $|\P|$ are not necessarily the same, since there could be a large set of points in the poset that all have a common strengthening. 


\subsection{Fullness}
\label{subsec:fullness}
Suppose that $\P$ is a forcing notion and $s$ is a set, with $\P, s \in H_\theta$, where $\theta$ is sufficiently large. Instead of just working with $H_\theta$ and its well order (as is done for the definition of proper forcing), the models used for defining subcompleteness will have the form $L_\tau[A] \supseteq H_\theta$ where $L_\tau[A] \models \ZFC^-$ such that $\tau>\theta$ is a (possibly singular) cardinal, and $A \subseteq \tau$. Typically, the convention is to set $N=L_\tau[A]$ for brevity. Rather than spell this out every time, write $N$ is a \emph{\textbf{suitable model}} for $\P,s$ above $\theta$.
Such $H_\theta$ will need to be large enough so that $N$ has the correct $\omega_1$ and $H_{\omega_1}$. One reason for working with these models is that such $N$ will naturally contain a well order of $H_\theta$, along with its Skolem functions. Additionally a benefit of working with such models is that $L_\tau[A]$ is easily definable in $L_\tau[A][G]$, if $G$ is generic, using $A$.

Our notation mostly follows Jensen but may be somewhat idiosyncratic. If $X \preccurlyeq N$ is a countable elementary substructure, we can take its transitive collapse and write $\N \cong X$. This gives rise to an elementary embedding 
$\sigma: \N \cong X \preccurlyeq N.$
Often we suppress the mention of the range of $\sigma$ and just write $\sigma: \N \prec N$.
In fact, it will not be quite enough for such an $\N$ to be transitive. We will want to be in a
situation where there is more than one elementary embedding of \(\N\) into \(N\).
This will be end up being possible if we require that the model \(\N\) is \emph{full}.

Note that, given $\sigma: \N \prec N$ where $\N$ is countable and transitive, the critical point $\cp(\sigma)$ is exactly $\omega_1^{\N} = \sigma^{-1}(\omega_1)$.

\begin{fact} \label{fact:CPofourEmbeddings} Suppose $\sigma: \N \prec N$ with $\N$ countable and $H_{\omega_1} \subseteq N$. Then $\cp(\sigma)=\omega_1^{\N}$ and $\sigma \rest H_{\omega_1}^{\N} = \id$. \end{fact}

Let $N = L_\tau[A]$, for some cardinal $\tau$ and set $A$, let $X$ be a set, and let $\delta$ be a cardinal. Our notation for the \emph{Skolem hull} of $\delta \cup X$ in $N$ is $\Sk{N}{\delta}{X}$. It is defined to be the smallest $Y \preccurlyeq N$ satisfying $X \cup \delta \subseteq Y$.

Toward defining fullness, the notion of regularity of transitive models needs to be defined.

\begin{definition} We say that $N$ is \emph{regular} in $M$ so long as $N \subseteq M$ and for all functions $f: x \To N$, where $x$ is an element of $N$ and $f \in M$, we have that $f``x \in N$. \end{definition}
The following lemma is meant to elucidate the significance of regularity as a kind of second-order
replacement scheme.

\begin{lemma} \label{lemma:regularityequiv}
$N$ is regular in $M$ if and only if $N = H_\gamma^M$, where $\gamma = \height(N)$ is a regular cardinal in $M$. 
\end{lemma}
\begin{proof}
For the backward direction, suppose that $N=H_\gamma^M$ where $\gamma = \height(N)$ is a regular cardinal in $M$. Then for all $f: x \To N$, with $x \in N$ and $f \in M$, certainly $f``x \in N$ as well.

For the forward direction, indeed $\gamma$ has to be regular in $M$ since otherwise $M$ would contain a cofinal function $f: \alpha \To \gamma$ where $\alpha < \gamma$. By the transitivity of $N$, this implies that $\alpha \in N$. Thus $\cup f``\alpha$ is in $N \models \ZFC^-$ by regularity, so $\gamma \in N$, a contradiction.
We have that $N \subseteq H_{\gamma}^M$ since $N$ is a transitive $\ZFC^-$ model, so the transitive closure of elements of $N$ may be computed in $N$ and thus have size less than $\gamma$, so they are in $H_\gamma^M$ as $\gamma \in M$. To show that $H_{\gamma}^M \subseteq N$, let $x \in H_{\gamma}^M$. We assume by $\in$-induction that $x \subseteq N$. Then there is a surjection $f: \alpha \twoheadrightarrow x$ where $\alpha < \gamma$, in $M$. Hence by regularity, $x = f``\alpha \in N$ as desired.
\end{proof}

\begin{definition} We say that $M$ is \emph{full} so long as $\omega \in M$, and there is a $\gamma$ such that $M$ is regular in $L_\gamma(M)$ and $L_\gamma(M) \models \ZFC^-$.
\end{definition}

Perhaps the property of fullness seems rather mysterious at first, but for the context of subcompleteness, it will be important to have; indeed, the fullness of a countable substructure $\N$ of $N$ guarantees that $\N$ is not pointwise definable, a necessary condition for the potential existence
of more than one embedding from $\N$ to $N$.

\begin{lemma} If $M$ is countable and full, then $M$ is not pointwise definable in the language of set theory. \end{lemma}
\begin{proof} Suppose toward a contradiction that $M$ is countable, full, and pointwise definable. By fullness there is some $L_\gamma(M) \models \ZFC^-$ such that $M$ is regular in $L_\gamma(M)$. By pointwise definability, for each element $m \in M$, we have attached to it some formula $\varphi(x)$ such that $M \models \varphi(m)$ uniquely, meaning that $\varphi(y)$ fails for every other element $y \in M$. Thus in $L_\gamma(M)$ we may define a function $f: \omega \cong M$, that takes the $n$th formula in the language of set theory to its unique witness in $M$. In particular we have that $L_\gamma(M)$ witnesses that $M$ is countable. However, this would allow $M$ to witness its own countability, since $M$ must contain $f$ as well by regularity. \end{proof} 

\subsection{Subcomplete Forcing} \label{subsec:subcomplete}
Subcompleteness is a weakening of the notion of completeness (that is, countable closure) of forcing notions. The following definitions and proofs are due to Jensen \cite[Ch.~3]{Jensen:2014}.

\begin{definition} A forcing notion $\P$ is \emph{complete} so long as
for any set $s$ and sufficiently large $\theta$ we have that once we are in the following situation: \begin{itemize}
	\item $N$ is a suitable model for $\P,s$ above $\theta$;
	\item $\sigma: \N \cong X \preccurlyeq N$ where $X$ is countable and $\N$ is transitive;
	\item $\sigma(\overline \theta, \overline{\P}, \overline s)=\theta, \P, s$.
\end{itemize}
and $\G$ is $\overline{\P}$-generic over $\N$ then there is a \emph{completeness condition} $p \in \P$ forcing that whenever $G \ni p$ is $\P$-generic, $\sigma ``\, \G \subseteq G$. 

In particular, below the condition $p$ we have that $\sigma$ lifts to an embedding $\sigma^*:\N[\G] \prec N[G]$.
We say that a $\theta$ as above \emph{verifies} the completeness of $\P$.
\end{definition}
The adjustment made to get subcomplete forcings is to not insist the the original embedding lifts in the forcing extension. Instead subcompleteness asks for an embedding, sufficiently similar
to the original one, which lifts to the extension, but may itself only exist in the extension.
As discussed in Section~\ref{subsec:fullness}, the domain of the embedding should be \emph{full} so as to not consistently rule out the possibility of the existence of multiple embeddings like this.

\begin{definition} \label{definition:SC}
A forcing notion $\P$ is \emph{\textbf{subcomplete}} so long as, for any set $s$ and 
for sufficiently large $\theta$ we have that whenever we are in \emph{the standard setup} i.e.: \begin{itemize}
	\item $N$ is a suitable model for $\P,s$ above $\theta$;
	\item $\sigma: \N \cong X \preccurlyeq N$ where $X$ is countable and $\N$ is full;
	\item $\sigma(\overline \theta, \overline{\P}, \overline s)=\theta, \P, s$;
\end{itemize}
then we have that if $\G$ is  $\overline{\P}$-generic over $\N$ then there is a \emph{subcompleteness condition} $p \in \P$ such that whenever $G \ni p$ is $\P$-generic, there is $\sigma' \in V[G]$ satisfying: \begin{enumerate}
	\item $\sigma': \N \prec N$;
	\item $\sigma'(\overline \theta, \overline{\P}, \overline s)=\theta, \P, s$;
	\item \label{item:skolemcompatibility} $\sk{N}{\delta(\P)}{\sigma'} = \Sk{N}{\delta(\P)}{X}$;
	\item  \label{item:sigmaprimelifts}$\sigma'``\, \G \subseteq G$.
\end{enumerate}
In other words and in particular, the subcompleteness condition $p$ forces that there is an embedding $\sigma'$ in the extension $V[G]$ which lifts, by (4), to an embedding $\sigma'^*:\N[\G] \prec N[G]$ in $V[G]$.

We say that such a $\theta$ as above \textit{verifies the subcompleteness} of $\P$.

Often we write $\delta$ instead of $\delta(\P)$ for the weight of $\P$ when there should be no confusion as to which poset $\P$ we are working with.
\end{definition}

\begin{remark}\label{remark:VerifyingSC}Let $\P$ be subcomplete as verified by $\theta$ and let $\theta'>\theta$. Then if $N$ is a suitable model for $\P$ above $\theta'$, it is not hard to see that the very same $N$ is a suitable model for $\P$ above $\theta$ as well. Thus $\P$ is subcomplete as verified by every $\theta' > \theta$. This tells us that $\P$ is subcomplete as long as it is verified by some $\theta$. So ``sufficiently large $\theta$" may be replaced with ``some $\theta$" in the definition of subcompleteness.\footnote{See \cite[Section 3.1 Lemma 2.4]{Jensen:2014}.} \end{remark}

\subsection{Subcomplete above $\mu$} \label{subsec:subcompleteabovemu}
The notion of subcompleteness above $\mu$ is an attempt to measure where exactly subcompleteness kicks in; in some sense it gives a level as to where the forcing fails to be subcomplete. 
The following definition is from Jensen \cite[Ch.~2 p.\ 47]{Jensen:2009fe}.
\begin{definition} 
Let $\mu$ be a cardinal. We say that a forcing notion $\P$ is \emph{subcomplete above $\mu$} so long as for every set $s$ and sufficiently large $\theta > \mu$, whenever we are in the standard setup (as in the definition of subcomplete forcing) in which \begin{itemize}
	\item $N$ is a suitable model for $\P, s$ above $\theta$;
	\item $\sigma:\N \cong X \preccurlyeq N$ where $X$ is countable and $\N$ is full;
	\item $\sigma(\overline{\theta}, \overline{\mu}, \overline{\P}, \overline s)=\theta, \mu, \P, s$;
\end{itemize}	
then, for any $\G \subseteq \overline{\P}$, there is a subcompleteness condition $p \in \P$ such that whenever $G \ni p$ is $\P$-generic, then there is $\sigma' \in V[G]$ as in the definition of subcompleteness satisfying the additional constraint that $\sigma' \rest H_{\overline \mu}^{\N} = \sigma \rest H_{\overline \mu}^{\N}$.

As usual, this means that in particular below the condition $p$ there is an embedding $\sigma'$ that lifts by (4) to an embedding $\sigma'^*:\N[\G] \prec N[G]$ in $V[G]$. We say that a $\theta$ as above \textit{verifies the subcompleteness above $\mu$ of $\P$}, and we may say that $\P$ is subcomplete above \(\mu\) if there is a $\theta$ that verifies its subcompleteness above $\mu$.
\end{definition}

\begin{remark} \label{remark:scaboveomega1} If $\P$ is subcomplete then $\P$ is subcomplete above $\omega_1$. \end{remark}
\begin{proof} By Fact \ref{fact:CPofourEmbeddings}, we have that for any elementary embedding such as $\sigma$ or $\sigma'$ from countable $\N$ to $N$, $\sigma' \rest H_{\overline{\omega_1}}^{\N} = \sigma \rest H_{\overline{\omega_1}}^{\N} = \text{id}$. \end{proof}

The following is a key property of posets that are subcomplete above \(\mu\).
\begin{theorem} \label{thm:nonewctblesubsetsofmu} If $\P$ is subcomplete above $\mu$ then $\P$ does not add new countable subsets of $\mu$. \end{theorem}
\begin{proof}
Suppose not. Let $\P$ be subcomplete above $\mu$. 
Fix a condition \(p\in\P\) and let $\dot{f}$ be a name such that
	$p \forces \left( \dot f: \check \omega \to \check{\mu} \right).$
Take $\theta > \mu$ large enough to verify the subcompleteness of $\P$, and let $N$ be a suitable model for $\P, \dot f$ above $\theta$. Moreover assume we are in the standard setup.
\begin{itemize}
	\item $\sigma: \N \cong X \preccurlyeq N$ where $\N$ is countable and full
	\item $\sigma(\overline \theta, \overline \mu, \overline{\P}, \overline p, \overline{\dot f})=\theta, \mu, \P, p, \dot f$.
\end{itemize}
Let $\G$ be $\overline{\P}$-generic over $\N$ with $\overline p \in \overline G$. By the subcompleteness of $\P$ above $\mu$ there is $q \in \P$ such that whenever $q \in G$ where $G$ is $\P$-generic, there is $\sigma' \in V[G]$ satisfying: \begin{enumerate}
	\item $\sigma': \N \prec N$;
	\item $\sigma'(\overline \theta, \overline \mu, \overline{\P}, \overline p, \overline{\dot f})=\theta, \mu, \P, p, \dot f$;
	\item $\sk{N}{\delta}{\sigma'} = \Sk{N}{\delta}{X}$;
	\item $\sigma'``\G \subseteq G$;
	\item $\sigma' \rest \overline \mu = \sigma \rest \overline \mu$.
\end{enumerate}
Let $\overline f = \overline{\dot f}^{\G}$ and $f = \dot f^G$. By (4) and (5), for each $n$, $f(n)=\sigma'(\overline f(n)) = \sigma(\overline f(n)),$ meaning that $f \in V$ since $\sigma,\overline{f} \in V$.
\end{proof}
Thus if $\P$ is subcomplete above $\mu$, then $\mu$'s cardinality, and even its cofinality, cannot be altered to be $\omega$ via $\P$.

It follows from Remark \ref{remark:scaboveomega1} and Theorem \ref{thm:nonewctblesubsetsofmu} that subcomplete forcing does not add new reals.

\subsection{Barwise Theory}
\label{subsec:BarwiseTheory}
In order to show that many posets are subcomplete, Jensen exploits Barwise Theory and techniques using countable admissible structures to obtain transitive models of infinitary languages.  The following is an outline of Jensen's notes on the subject \cite[Ch.~1 \& 2]{Jensen:2014}. 

\begin{definition} A transitive structure $\M$ is \emph{admissible} if it models the axioms of \textsf{Kripke-Platek Set Theory} (\textsf{KP}) which consists of the axioms of \textsf{Empty Set}, \textsf{Pairing}, \textsf{Union}, $\Sigma_0$-\textsf{Collection}, and $\Sigma_0$-\textsf{Separation}. \end{definition}

Jensen also makes use of models of $\textsf{ZF}^-$ that are not necessarily well-founded.
\begin{definition} Let $\mathfrak A = \langle A, \in_{\mathfrak A}, B_1, B_2, \dots \rangle$ be a (possibly ill-founded) model  of $\textsf{ZF}^-$, where $\mathfrak A$ is allowed to have predicates other than $\in$. The \emph{well-founded core} of $\mathfrak A$, denoted $\wfc(\mathfrak A)$, is the restriction of $\mathfrak A$ to the set of all $x \in A$ such that $\in_{\mathfrak A} \cap \ \mathcal C(x)^2$ is well founded, where $\mathcal C(x)$ is the closure of $\{x\}$ under $\in_{\mathfrak A}$. A model $\mathfrak A$ of $\textsf{ZF}^-$ is \emph{solid} so long as $\wfc(\mathfrak A)$ is transitive and $\in_{\wfc(\mathfrak A)}=\in \cap \wfc(\mathfrak A)^2$. \end{definition}

Jensen \cite[Section 1.2]{Jensen:2014} notes that every consistent set of sentences in $\textsf{ZF}^-$ has a solid model, and if $\mathfrak A$ is solid, then $\omega \subseteq \wfc(\mathfrak A)$. In addition,

\begin{fact}[{\cite[Ch.~1 Lemma 21]{Jensen:2014}}] If a model $\mathfrak A$ of $\textsf{\textup{ZF}}^-$ is solid, then $\wfc(\mathfrak A)$ is admissible. \end{fact}

Barwise creates an $\M$-finite predicate logic, a first order theory in which arbitrary, but $\M$-finite, disjunctions and conjunctions are allowed.
\begin{definition} Let $\M$ be a transitive structure with potentially infinitely many predicates. A theory defined over $\M$ is $\M$-\emph{finite} so long as it is in $\M$. A theory is $\Sigma_1(\M)$, also known as \emph{$\M$-recursively enumerable} or $\M$-$re$, if the theory is $\Sigma_1$-definable, with parameters from $\M$. \end{definition}
Of course this may be generalized to the entire usual Levy hierarchy of formulae, but only $\Sigma_1(\M)$ is needed in this paper. If $\mathcal L$ is a $\Sigma_1(\M)$-definable language or theory, the rough idea is that to check whether a sentence is in $\mathcal L$, one should imagine enumerating the formulae of $\mathcal L$ to find a sentence and a witness to it in the structure $\M$.

\begin{definition} \label{def:InTheoriesAndBasicAxioms} Let $\M$ be admissible. An infinitary axiomatized theory in $\M$-finitary logic $\mathcal L=\mathcal L(\M)$, with a fixed predicate $\dot \in$ and \emph{special constants} denoted $\underline x$ for elements $x \in \M$, is called an \textit{$\in$-theory} over $\M$. The underlying axioms for these $\in$-theories will always involve $\ZFC^-$ and some basic axioms ensuring that $\dot \in$ behaves nicely; 
the \textsf{Basic Axioms} are: \begin{itemize}
	\item \textsf{Extensionality}
	\item A statement positing the extensionality of $\dot \in$, which is a scheme of formulae defined for each member of $M$. For each $x \in M$,
	$\forall v \left( v \ \dot \in\, \underline x \iff \bigdoublevee_{z \in x} v= \underline z \right).$
\end{itemize} Here $\bigdoublevee$ denotes an infinite disjunction in the language.\end{definition}

In the above definition, it should be clarified that it is possible to consider the same $\in$-theory defined over different admissible structures; if $\M, \M'$ are both admissible, then we can consider both $\mathcal L(\M)$ and $\mathcal L(\M')$. The distinction is only as to where the special constants come from. 

An important fact ensured by our \textsf{Basic Axioms} is that the interpretations of these special constants in any solid model of the theory are the same as in $\M$:

\begin{fact}[{\cite[Ch.~2, Lemma 9]{Jensen:2014}}] \label{fact:PointofBasicAxioms} Let $\M$ be admissible and let $\mathcal L$ be an $\in$-theory over $\M$. Let $\mathfrak A$ be a solid model of $\mathcal L$. Then for all $x \in \M$, we have that $\underline{x}^{\mathfrak A} = x \in \wfc(\mathfrak A)$. \end{fact}

Jensen uses the techniques of Barwise to come up with a proof system in this context, in which consistency of $\in$-theories can be discussed. In particular, the semantics is sound for this syntax: if there is a model of an \(\in\)-theory, then the theory is consistent. 
\begin{fact}[\emph{Barwise Correctness}] \label{fact:correctness} 
Let $\mathcal L$ be an $\in$-theory. If $X$ is a set of $\mathcal L$-statements and $\mathfrak A \models X$, then $X$ is consistent. \end{fact}

Furthermore, the proof system is absolute enough that consistency statements are downward absolute. In particular, it will be useful to know that if a theory is consistent in a forcing extension, then it is consistent in the ground model. Compactness and completeness are also shown, relativized to the $\M$-finite predicate logics that are used here; solid models of consistent $\Sigma_1(\M)$ $\in$-theories are produced for countable admissible structures $\M$. 

\begin{fact}[\emph{Barwise Completeness}] \label{fact:completeness} Let $\M$ be a countable admissible structure. Let $\mathcal L$ be a consistent $\Sigma_1(\M)$ $\in$-theory with $\mathcal L \vdash \textsf{\textup{ZF}}^-$. Then $\mathcal L$ has a solid model $\mathfrak A$ such that $\Ord \cap \wfc(\mathfrak A) = \Ord \cap \M.$ \end{fact}

The following definition generalizes the concept of fullness.
\begin{definition} We say that $M$ is \emph{almost full} so long as $\omega \in M$ and there is a solid $\mathfrak A \models \ZFC^-$ with $M \in \wfc(\mathfrak A)$ and $M$ is regular in $\mathfrak A$. \end{definition}
Clearly if $M$ is full, then $M$ is almost full.

A useful technique when showing a particular forcing is subcomplete is to be able to transfer the consistency of $\in$-theories over one admissible structure to another. First we define what it means for an embedding to be cofinal.

\begin{definition} Let $\delta_N$ be the least $\delta$ such that $L_\delta(N)$ is admissible. \end{definition}

\begin{definition} We say that an elementary embedding $\sigma: M \prec N$ is \emph{cofinal} so long as for each $x \in N$ there is some $u \in M$ such that $x \in \sigma(u)$. 

Let $\delta \in M$ be a cardinal. We say that $\sigma$ is \emph{$\delta$-cofinal} so long as every such $u$ has size less than $\delta$ as computed in $M$. \end{definition}

The following is a fairly straightforward exercise that sheds light on some of the advantages of $\alpha$-cofinal embeddings.

\begin{observation}[{\cite[Cor. 3.7]{Jensen:2014}}] \label{observation:regularitySups}
Let $\sigma:M \prec N$ $\delta$-cofinally, where $\delta \in M$ is a cardinal in $M$. Let $\nu \geq \delta$ be a regular cardinal in $M$. Then $\sigma(\nu)=\sup \sigma``\nu$ and $H_{\sigma(\nu)}^N=\bigcup_{u \in H_\nu^M}\sigma(u)$.
\end{observation}

The Transfer Lemma is an upward absoluteness statement, transferring the consistency of a $\Sigma_1$ $\in$-theory over an almost full model upward, via cofinal elementary embeddings. 

\begin{fact}[\emph{The Transfer Lemma} {\cite[Lemma 4.5]{Jensen:2014}}]\label{fact:Transfer} Let $N_1$ be almost full and $k: N_1 \prec N_0$ cofinally, for some structure $N_0$. Suppose that we have a $\Sigma_1(\langle N_1; \, p_1, \dots, p_n\rangle)$  $\in$-theory $\mathcal L=\mathcal L(L_{\delta_{N_1}}(N_1))$  for $p_1, \dots, p_n \in N_1$.

Moreover, suppose $\mathcal L(L_{\delta_{N_0}}(N_0))$ is also defined and is a $\Sigma_1(\langle N_0; \, k(p_1), \dots, k(p_n)\rangle)$ $\in$-theory as well. 

Then, 
if $\mathcal L(L_{\delta_{N_1}}(N_1))$ is consistent, it follows that $\mathcal L(L_{\delta_{N_0}}(N_0))$ is also consistent. \end{fact} 

\subsection{Liftups}
\label{subsec:Liftups}
The following definitions are meant to describe a method used to obtain useful embeddings, outlining a technique that is ostensibly the ultrapower construction. The way that these embeddings are constructed facilitate the use of Barwise theory to obtain the consistency of the existence of the kind of embedding required in the definition of subcompleteness. Refer to \cite[Ch.~3]{Jensen:2014} for all of the general definitions and theorems, the specific relevant necessary pieces are given here.

\begin{definition} Let $\alpha > \omega$ be a regular cardinal in $\N$. Let 
	$\overline \sigma$ be a cofinal embedding from $H^{\N}_\alpha$ to some transitive set. By a \emph{transitive liftup} of $\langle \N, \overline \sigma \rangle$ we mean a pair $\langle N_* , \sigma_* \rangle$ such that 
\begin{itemize} 
	\item $N_*$ is transitive;
	\item $\sigma_*:\N \prec N_*$ $\alpha$-cofinally;
	\item $\sigma_* \upharpoonright H_{\alpha}^{\N}= \overline \sigma$. \qedhere
\end{itemize}	
\end{definition}

Jensen \cite[Lemma 3.1]{Jensen:2014} shows  that transitive liftups, if they exist, are determined up to isomorphism. To do this he uses the following characterization of transitive liftups:
\begin{lemma}\label{lemma:liftupchar} Let $\sigma: \N \prec N$. Then,
$\sigma$ is $\alpha$-cofinal if and only if elements of $N$ are of the form $\sigma(f)(\beta)$ for some $f: \gamma \to \overline N$ where $\gamma < \alpha$ and $\beta < \sigma(\gamma)$. \end{lemma} 
\begin{proof} 
For the forward direction, let $x \in N$, and take $u \in \overline N$ with $x \in \sigma(u)$ such that $|u| <\alpha$ in $\overline N$. Let $|u|=\gamma$, and take $f:\gamma \to u$ a bijection in $\overline N$. Then $\sigma(f):\sigma(\gamma) \to \sigma(u)$ is also a bijection in $N$ by elementarity. Since $x \in \sigma(u)$ we also have that $x$ has a preimage under $\sigma(f)$, say $\beta$. So $\sigma(f)(\beta)=x$ as desired.

For the backward direction, let $x=\sigma(f)(\beta)$ be an element of $N$, for $f:\gamma \to \overline N$ where $\gamma < \alpha$ in $\overline N$ and $\beta < \sigma(\gamma)$. Define $u = f``\gamma$. Then in $\overline N$ we have that $|u|<\alpha$. In addition we have that $x \in \sigma(u)$, since $\sigma(u)$ is in the range of $\sigma(f)$, where $x$ lies.
\end{proof}

Furthermore, Jensen shows that transitive liftups exist, provided that some extension of the original embedding exists, and have a nice factorization property.

\begin{fact}[\emph{Interpolation} {\cite[Lemma 5.1]{Jensen:2014}}]\label{fact:Interpolation} Let $\sigma:\N \prec N$ and let $\alpha \in \N$ be a regular cardinal. Then: \begin{enumerate}
	\item The transitive liftup $\langle N_*, \sigma_* \rangle$ of $\langle \N, \sigma \rest H^{\N}_\alpha \rangle$ exists.
	\item There is a unique $k_*:N_* \prec N$ such that $k_* \circ \sigma_* = \sigma$ and $k_* \rest \bigcup \sigma `` H^{\N}_\alpha = \id$.
\end{enumerate}
\end{fact} 

For the next useful lemma, a more general notion of liftups needs to be defined - where the model produced is not necessarily transitive.
\begin{definition} Let $\mathfrak A$ be a solid model of $\ZFC^-$ and let $\tau \in \wfc(\mathfrak A)$ be an uncountable cardinal in $\mathfrak A$. Let $\sigma : H_{\tau}^\mathfrak A \To A$ cofinally, where $A$ is transitive. Then by a \textit{liftup} of $\langle \mathfrak A, \sigma\rangle$, we mean a pair $\langle {\mathfrak A}_*, \sigma_* \rangle$ such that 
\begin{itemize}
	\item $\sigma_* \supseteq \sigma$;
	\item ${\mathfrak A}_*$ is solid;
	\item $\sigma_*: \mathfrak A \To_{\Sigma_0} {\mathfrak A}_*$ $\tau$-cofinally;
	\item $A \subseteq \wfc({\mathfrak A}_*)$. \qedhere
\end{itemize}
\end{definition}
\begin{fact}[{\cite[Lemma 3.3]{Jensen:2014}}]\label{fact:solidliftup} Let $\mathfrak A$ be a solid model of $\ZFC^-$. Let $\tau > \omega$, $\tau \in \wfc(\mathfrak A)$, and let $\sigma$ be a cofinal embedding from $H_{\tau}^{\mathfrak A}$ to some transitive set. Then $\langle \mathfrak A, \sigma \rangle$ has a liftup $\langle \mathfrak A_*, \sigma_* \rangle$.
\end{fact}

The following lemma states that transitive liftups of full models are almost full.

\begin{lemma}[{\cite[Lemma 2.1]{Jensen:2014}}] \label{lemma:liftupfull} Let  $N= L_\tau[A]$ and $\sigma: \N \prec N$ where $\N$ is full, and let $\overline \sigma=\sigma\rest H_\alpha^{\N}$ for some $\alpha$. Suppose that $\langle N_*, \sigma_* \rangle$ is a transitive liftup of $\langle \N, \overline \sigma \rangle$. Then $N_*$ is almost full. \end{lemma}
\begin{proof}
	Let $L_\gamma(\N)$ witness the fullness of $\N$. We will now apply Interpolation (Fact \ref{fact:Interpolation}) to $\mathfrak A= L_\gamma(\N)$, which makes sense since certainly $\mathfrak A$ is a transitive model of $\ZFC^-$. Additionally, by Lemma \ref{lemma:regularityequiv} we have that $\N = H_\eta^{\mathfrak A}$, where $\eta = \height(\N)$. 
	Since $\langle N_*, \sigma_* \rangle$ is a transitive liftup, we have that $\sigma_*: H_\eta^{\mathfrak A} \prec N_* \text{ cofinally,}$ where $N_*$ is transitive. Thus since $\mathfrak A$ is transitive, $\langle \mathfrak A, \sigma_* \rangle$ has a liftup $\langle \mathfrak A_*, {\sigma_*}_* \rangle$, where $\mathfrak A_* \models \ZFC^-$ since $\mathfrak A$ does, $\mathfrak A_*$ is solid, where ${\sigma_*}_* : \mathfrak A \prec \mathfrak A_* \text{ $\eta$-cofinally.}$ 
	We have that  $N_* \subseteq \wfc({\mathfrak A}_*)$ and $\eta_*= {\sigma_*}_*(\eta)=\height(N_*)$ is regular since $\eta$ is. Furthermore, we will show that $N_* = H_{\eta_*}^{\mathfrak A_*}$, completing the proof:
	
	Certainly it is the case that $N_* \subseteq H_{\eta_*}^{\mathfrak A^*}$. 
	But if $x \in H_{\eta_*}$ in $\mathfrak A_*$, then by regularity we have that $x \in {\sigma_*}_*(u)$ in $\mathfrak A_*$, where $u \in \mathfrak A$, and $|u| < \eta$ in $\mathfrak A$. Let $v=u \cap H_\eta$ in $\mathfrak A$. Then $v \in \N$, since $\N$ is regular in $\mathfrak A$. But then $x \in \sigma_*(v) \in N_*$. So $x \in N_*$.
\end{proof}

\section{Generalized Diagonal Prikry Forcing}
\label{sec:GenDiagonalPrikryForcing}
Generalized diagonal Prikry forcing is designed to add a point below every measurable cardinal in an infinite discrete set of measurable cardinals. In this section it is shown that such forcing notions are subcomplete.
\begin{definition}
Let $D$ be an infinite discrete set of measurable cardinals, meaning a set of measurable cardinals that does not contain any of its limit points. For $\kappa \in D$ let $U(\kappa)$ be a normal measure on $\kappa$, and let $\U$ denote the sequence of the $U(\kappa)$'s.

Define $\D=\D(\U)$, \emph{generalized diagonal Prikry forcing} from the list of measures $\U$, by taking conditions of the form  
$( s, A )$ satisfying the following:
\begin{itemize}
	\item The \textit{stem} of the condition, $s$, is a function with domain in $[D]^{<\omega}$ taking each measurable cardinal $\kappa \in \dom(s)$ to some ordinal $s(\kappa) < \kappa$.
	\item The \textit{upper part} of the condition, $A$, is a function with domain $D \setminus {\dom(s)}$ taking each measurable cardinal $\kappa \in \dom(A)$ to some measure-one set $A(\kappa) \in U(\kappa)$.
\end{itemize}
We say that $( s, A ) \leq ( t, B )$ so long as 
\begin{itemize}
	\item $s \sqsupseteq t$, i.e., $s$ is an end extension of $t$;
	\item the points in $s$ not in $t$ come from $B$, i.e., for all $\kappa \in \dom(s) \setminus \dom(t)$, $s(\kappa) \in B(\kappa)$;
	\item for all $\kappa \in \dom(t)$, $A(\kappa) \subseteq B(\kappa)$.
\end{itemize}
If $G$ is a generic filter for $\D$, then its associated $\D$-generic sequence is \[S = S_G = \bigcup \set{ s }{ \exists A \ ( s, A ) \in G }.\qedhere\]
\end{definition}

Here the definition of $\D(\U)$ differs from that of generalized Prikry foricing as given by Fuchs~\cite{Fuchs:2005kx}. The main difference is that here only one point is added below each measurable cardinal $\kappa \in D$, which is done for simplicity's sake. It is not hard to see that the following theorem also shows that the forcing adding countably many points below each measurable cardinal in $D$ is subcomplete. Adding countably many points below each measurable cardinal in $D$ would collapse the cofinality of each $\kappa \in D$ to be $\omega$, as one expects of a Prikry-like forcing. 

Also in the above definition it hasn't been enforced that the stem a condition consists only of ordinals that are wedged between successive measurables in $D$; ie. for $\kappa \in D$, it is not explicitly insisted that $s(\kappa) \in \left[\sup(D\cap \kappa), \kappa\right)$. However, it is dense on a tail of the generic filter of $\D(\U)$ for the conditions to be that way, since conditions may be strengthened by restricting their upper parts to a tail. 
Thus such a restriction may be freely added to the following genericity condition on $\D(\U)$.

\begin{fact}[Fuchs~{\cite[Thm.~1]{Fuchs:2005kx}}] \label{fact:diagprikrymathias} Let $D$ be an infinite discrete set of measurable cardinals in $M$, with $\U$ a corresponding list of measures $\langle U(\kappa) \;|\; \kappa \in D \rangle$. Then an increasing sequence of ordinals $S = \seq{ S(\kappa) }{ \kappa \in D }$ in $V$, where for each $\kappa \in D$, we have that $$\sup(D \cap \kappa) < S(\kappa) < \kappa,$$ is a $\D(\U)$-generic  sequence if and only if for all $\mathcal X = \seq{ X_\kappa \in U(\kappa) }{ \kappa \in D } \in M$, the set $\set{ \kappa \in D }{ S(\kappa) \notin X_\kappa }$ is finite.
\end{fact}

The above genericity criterion on generalized diagonal Prikry forcing is similar to the Mathias genericity criterion for Prikry forcing.

\begin{theorem} \label{theorem:main} Let $D$ be an infinite discrete set of measurable cardinals and let $\U = \seq{ U(\kappa) }{\kappa \in D }$ be a list of normal measures associated to the
measurable cardinals in $D$. 
The generalized diagonal Prikry forcing $\D=\D(\U)$ is subcomplete. \end{theorem}
\begin{proof} 
Let $\theta > \delta(\D) =\delta$ be large enough, so that $[\delta]^{<\omega_1} \in H_\theta$. 

First of all, it must be the case that $\delta$, the weight of $\D$, is larger than all of the measurable cardinals in $D$.
\begin{claimno} $\delta \geq \sup D$. \end{claimno}
\begin{proof}[Pf.] Suppose instead that there is a dense set $E \subseteq \D$ such that $\sup D \geq \kappa^* >|E|$ for some $\kappa^* \in D$. Then for each condition $(s,A) \in E$ either $\kappa^* \in \dom s$ or $\kappa^* \in \dom A$. So taking $E^*=\set{ (s, A) \in E }{ \kappa^* \in \dom s } \subseteq E$, since $|E^*| < \kappa^*$ as well, there is an $\alpha<\kappa^*$ such that $\sup_{(s,A) \in E^*} s(\kappa^*)=\alpha$. Let $p=(t, B) \in \D$ be defined so that $t(\kappa^*)=\alpha$ and $B(\kappa)=\kappa$ for all $\kappa \in D \setminus \{ \kappa^*\}$. Then $p$ cannot be strengthened by any condition in $E$ since $\kappa^*$ is not in any of the stems of conditions in $E$. So dense subsets of $\D$ must have size at least $\sup D$.\end{proof}

Let $\nu = \delta^+$. Let $\kappa(0)$ be the first measurable cardinal in $D$. 
In order to show that $\D$ is subcomplete, let $c$ be a set and suppose we are in the standard setup: \begin{itemize}
	\item $N$ is a suitable model for $\D, c$ above $\theta$;
	\item $\sigma: \N \cong X \preccurlyeq N$ where $X$ is countable and $\N$ is full;
	\item $\sigma(\overline \theta, \overline{\D}, \overline{\U}, \overline c)=\theta, \D, \U, c$.
\end{itemize}
By our requirement on $\theta$ being large enough, we've ensured that $N$ is closed under countable sequences of ordinals less than $\delta$.

In what follows, we will be taking a few different transitive liftups of restrictions of $\sigma$, and it will useful to keep track of embeddings between $\N$ and $N$ pictorially. Figure \ref{figure:N} shows the situation we are currently in.

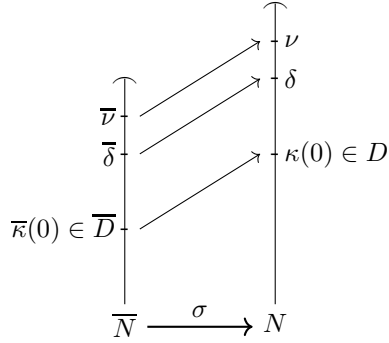
\begin{figure}[h!] 
\begin{tikzpicture}
\draw (-1,0)--(-1,3);
\node [below] at (-1,0) {$\N$};
\node at (-1,2.95) {$\frown$};

\node [left] at (-1,1) {$\overline \kappa(0) \in \overline D$};
\node at (-1,1) {-};
\node [left] at (-1,2) {$\overline \delta$};
\node at (-1,2) {-};
\node [left] at (-1,2.5) {$\overline \nu$};
\node at (-1,2.5) {-};

\draw [->] (-.8,1)--(.8,2);
\draw [->] (-.8,2)--(.8,3);
\draw [->] (-.8,2.5)--(.8,3.5);

\node [right] at (1,2) {$\kappa(0) \in D$};
\node at (1,2) {-};
\node [right] at (1,3) {$\delta$};
\node at (1,3) {-};
\node [right] at (1,3.5) {$\nu$};
\node at (1,3.5) {-};

\draw (1,0)--(1,4);
\node [below] at (1,0) {$N$};
\node at (1,3.95) {$\frown$};

\draw [->, thick] (-.7,-0.3)--(.7,-0.3);
\node [above] at (0,-0.3) {$\sigma$};
\end{tikzpicture}
\caption{Diagram depicting $\sigma$.}\label{figure:N} 
\end{figure}

Here we place bars on everything that is relevant on the $\N$ side of the embedding. In particular, $\sigma(\overline \delta)=\delta$, $\overline \nu = {\overline \delta^+}^{\N}$, $\overline D$ is the discrete set of measurables in $\N$ that each measure in $\overline \U$ comes from, and $\overline \kappa(0)$ is the first measurable in $\overline D$, in the sense of $\N$. 
Toward showing that $\D$ is subcomplete, we are additionally given some $\G \subseteq \overline{\D}$ that is generic over $\N$. Rather than working with $\G$, we will work with $\S = \seq{ \S(\overline \kappa) }{ \overline \kappa \in \overline D }$, its associated $\overline \D$-generic sequence. 

Letting $C=\Sk{N}{\delta}{X}$, in order to show that $\D$ is subcomplete, we will be done if we can show the following Main Claim is satisfied.

\begin{claim} There is a $\D$-generic sequence $S$ and a map $\sigma' \in V[S]$ such that: \begin{enumerate}
	\item \label{item:mainclaimsigma'elem} $\sigma': \N \prec N$;
	\item \label{item:mainclaim2} $\sigma'(\overline \theta, \overline{\D}, \overline{\U}, \overline c)=\theta, \D, \U, c$;
	\item \label{item:MainClaimC=Sk} $\sk{N}{\delta}{\sigma'} = C$;
	\item \label{item:mainclaimlifts} $\sigma' `` \S \subseteq S$.
\end{enumerate} \end{claim}
\begin{proof}[Pf.]  Ultimately this proof amounts to showing that a certain 
$\in$-theory, $\mathcal L'$, which posits the existence of such a $\sigma'$, is consistent. Such an embedding $\sigma'$ cannot exist in $V$ (otherwise the forcing $\D$ would be countably closed, which we know is not the case), but it will exist in an extension of the form $V[S]$, where $S$ is $\D$-generic sequence that we will find later.
Once we have such a suitable $V[S]$, we will find an appropriate admissible structure in $V[S]$ for which the $\in$-theory $\mathcal L'$, defined below, has a model.

\begin{definition}[$\mathcal L'$] \label{definition:T} Let $S$ be a $\D$-generic sequence. For an admissible structure $\M$ with $S, \overline S, \sigma, \N, N, \theta, \D, \U, c \in \M$ we define the $\in$-theory $\mathcal L'(\M)$ as follows.
\begin{description}
	\item[predicate] $\in$.
	\item[constants] $\dot{\sigma}, \underline x$ for $x \in \M$.
	\item[axioms] \begin{itemize} \item $\ZFC^-$ and \textsf{Basic Axioms};
		\item $\dot \sigma : \overline{\underline N} \prec \underline N$;
		\item $\dot{\sigma}(\overline{\underline{\theta}}, \overline{\underline{\D}}, \overline{\underline{\U}}, \overline{\underline c})=\underline{\theta}, \underline{\D}, \underline{\U}, \underline{c}$;
		\item $\sk{\underline N}{\underline{\delta}}{\dot \sigma} = \sk{\underline N}{\underline \delta}{\underline \sigma}$;
		\item $\dot \sigma ``\overline{\underline S} \subseteq \underline S$.
	\end{itemize}
\end{description}
\end{definition}

The $\in$-theory $\mathcal L'$ is $\Sigma_1(\M)$, since all of the axioms are $\M$-finite except for the \textsf{Basic Axioms}, which altogether are $\M$-$re$. Recall that by Fact \ref{fact:solidliftup}, the basic axioms make sure that any model $\mathfrak A$ of $\mathcal L'(\M)$ will have the property that for any $x \in \M$, $\underline x^{\mathfrak A}=x$. In this and other $\in$-theories to come, we will make explicit what our extra constants are (with $\mathcal L'$ we just have one extra constant, which we denote $\dot \sigma$).

We need to find an appropriate $\D$-generic sequence $S$ and a suitable admissible structure $\M$ containing $S$ so that $\mathcal L'(\M)$ is consistent. To do this we use transitive liftups and Barwise theory. Transitive liftups will provide approximations to the embedding we are looking for, and we will rely on Barwise Completeness (Fact \ref{fact:completeness}) to obtain the existence of a model with our desired properties. 

Toward this end, let's take what will turn out to be our first transitive liftup, which is in some sense ensuring (\ref{item:MainClaimC=Sk}) of our main claim.

Let $k_0 : N_0 \cong C$ where $N_0$ is transitive, and set $\sigma_0 = k_0^{-1} \circ \sigma$ and $\sigma_0(\overline \theta, \overline{\D}, \overline{\U}, \overline c) = \theta_0, \D_0, \U_0, c_0.$
Since $\delta \subseteq C$ and $N_0$ is transitive, $\sigma_0(\overline \delta)=\delta$. 

Indeed $N_0$ is actually a transitive liftup.\footnote{As Jensen {\cite[Lemma 5.3]{Jensen:2014}} explains, this is exactly why the third requirement of subcompleteness involves this type of Skolem hull - these liftups can be recovered.}

\begin{claimno} \label{claim:N0isliftupofN} $\langle N_0, \sigma_0 \rangle$ is the transitive liftup of $\langle \N, \sigma \upharpoonright H_{\overline \nu}^{\N} \rangle$. \end{claimno}
\begin{proof}[Pf.] Recall that $\nu=\delta^+$, and $\overline \nu={\overline \delta^+}^{\N}$. It must be shown that the embedding $\sigma_0: \overline N \prec N_0$ is $\overline \nu$-cofinal and that $\sigma_0 \upharpoonright H_{\overline \nu}^{\N}=\sigma \upharpoonright H_{\overline \nu}^{\overline N}$. 

To see that $\sigma_0$ is $\overline \nu$-cofinal, let $x \in N_0$. Then $k_0(x) \in C = \Sk{N}{\delta}{X}$ so $k_0(x)$ is uniquely $N$-definable from some $\xi < \delta$ and $\sigma(\overline z)$ where $\overline z \in \N$. In other words, we may say that $k_0(x)$ is exactly that  $y$ satisfying the property that $N \models \varphi(y, \xi, \sigma(\overline z))$ for some $\varphi$. Let $u \in \N$ be defined as 
	$$u=\set{ w \in \N }{\text{$w$ is unique satisfying $\N \models \varphi(w, \zeta, \overline z)$ for some $\zeta < \overline \delta$}}.$$
Certainly $u$ is non-empty by elementarity, as we know $\sigma(u)$ is nonempty, for example $k_0(x) \in \sigma(u)$.
Furthermore, $|u| \leq \overline \delta < \overline \nu$ since every $w \in u$ comes with a (unique) corresponding $\zeta<\overline \delta$.
Thus $x \in k_0^{-1}(\sigma(u))=\sigma_0(u)$ with $|u| < \overline \nu$ in $\N$. So $\sigma_0$ is $\overline \nu$-cofinal as desired. 

For the second part of the claim, recall that $X=\sigma``\N$ is the range of $\sigma$, and make the simple observation that $\sigma``H_{\overline \nu}^{\N} \subseteq X \cup \delta \subseteq C$. Thus $k_0^{-1} \upharpoonright \sigma``H_{\overline \nu}^{\overline N} = \text{id}$.
Therefore $\sigma_0 \rest H_{\overline \nu}^{\N}=\sigma \rest H_{\overline \nu}^{\overline N}$, finishing the proof of Claim \ref{claim:N0isliftupofN}.
\end{proof}

Since $\overline \nu$ is regular in $\N$, so is $\nu_0=\sigma_0(\overline \nu)=\sup \sigma_0``\overline \nu$ in $N_0$, where the last equality follows since $\sigma_0$ is $\overline \nu$-cofinal (see Observation \ref{observation:regularitySups}).
By Interpolation (Fact \ref{fact:Interpolation}), we may say that the transitive collapse embedding $k_0$ may be defined by 
	$k_0: N_0 \prec N \text{ where } k_0 \circ \sigma_0=\sigma \text{ and } k_0 \rest \nu_0 = \id.$
In particular, $\nu_0$ is the critical point of the embedding $k_0$. We illustrate the situation up to this point with Figure \ref{figure:N0N}. 

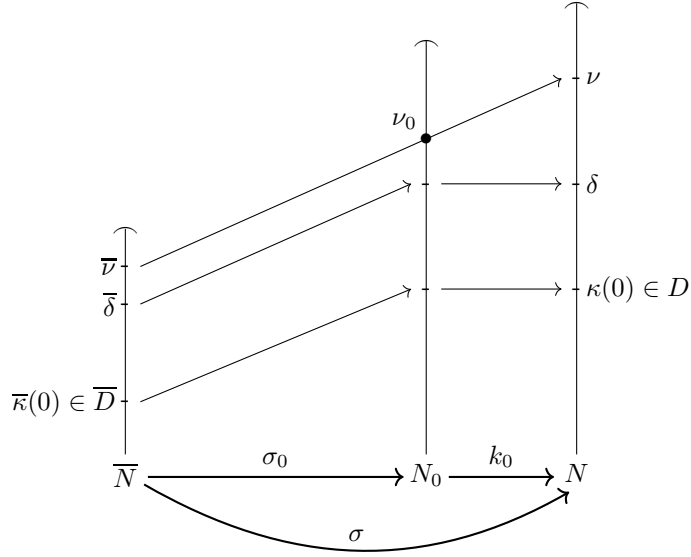
\begin{figure}[h!]
\begin{tikzpicture}
\draw (-3,0)--(-3,3);
\node [below] at (-3,0) {$\N$};
\node at (-3,2.95) {$\frown$};

\node [left] at (-3,0.7) {$\overline \kappa(0) \in \overline D$};
\node at (-3,0.7) {-};
\node [left] at (-3,2) {$\overline \delta$};
\node at (-3,2) {-};
\node [left] at (-3,2.5) {$\overline \nu$};
\node at (-3,2.5) {-};

\draw [->] (-2.8,0.7)--(0.8,2.2);
\draw [->] (-2.8,2)--(0.8,3.6);

\node at (1,2.2) {-};

\node at (1,3.6) {-};
\node [above left] at (1,4.2) {$\nu_0$};
\node at (1,4.2) {$\bullet$};

\draw (1,0) -- (1,5.5);
\node [below] at (1,0) {$N_0$};
\node at (1,5.45) {$\frown$};

\draw [->] (1.2,2.2)--(2.8,2.2);
\draw [->] (1.2,3.6)--(2.8,3.6);
\draw [->] (-2.8,2.5)--(2.8,5);

\node [right] at (3,2.2) {$\kappa(0) \in D$};
\node at (3,2.2) {-};
\node [right] at (3,3.6) {$\delta$};
\node at (3,3.6) {-};
\node [right] at (3,5) {$\nu$};
\node at (3,5) {-};

\draw (3,0)--(3,6);
\node [below] at (3,0) {$N$};
\node at (3,5.95) {$\frown$};

\draw [->, thick] (-2.7,-0.3)--(0.7,-0.3);
\draw [->, thick] (1.3, -0.3)--(2.7,-0.3);
\node [above] at (-1,-0.3) {$\sigma_0$};
\node [above] at (2,-0.3) {$k_0$};
\draw[->,thick] (-2.75,-0.4) to [out = -30, in =-150] node[above]{$\sigma$} (2.9, -0.5);
\end{tikzpicture}
 \caption{Diagram depicting the liftup $\langle N_0, \sigma_0\rangle$.} \label{figure:N0N}
\end{figure}

The embedding $\sigma_0$ has a nice property: $\sk{N}{\delta}{\sigma_0}=\Sk{N}{X}{\delta}= C.$ This is reminiscent of (\ref{item:MainClaimC=Sk}) in the Main Claim. 
Importantly, we still need to find a way to extend the generic sequence $\S$ to a $\D$-generic sequence over $N$. 

We will now define another $\in$-theory, which we will call $\mathcal L_*$. 
This will assist us in obtaining the diagonal Prikry extension $V[S]$ we need to satisfy our main claim. In order to do this, we will take another transitive liftup and apply the Transfer Lemma (Fact \ref{fact:Transfer}), in order to see that this new $\in$-theory is consistent over an admissible structure on $N_0$. 

We give a general definition of the $\in$-theory we will work with. Since we will be referring to basically the same $\in$-theory over two different transitive liftups, the reader should think of $``*"$ in the subscript as a kind of placeholder for some transitive liftup.

\begin{definition}[$\mathcal L$] \label{definition:L} Let $\langle N_*, \sigma_* \rangle$ be a transitive liftup of $\N$ along with some reasonable restriction of $\sigma$, ie. the liftup of $\langle \N, \sigma \rest H_{\alpha}^{\N} \rangle$, where $\alpha \geq \overline \kappa(0)$ is regular in $\N$, and say $\sigma_*(\overline \theta, \overline{\D}, \overline{\U}, \overline c) = \theta_*, \D_*, \U_*, c_*.$ 

Define the infinitary $\in$-theory $\mathcal L(L_{\delta_{N_*}}(N_*))=\mathcal L_*$ as follows.\footnote{Recall that $\delta_{N_*}$ is the least ordinal for which $L_{\delta_{N_*}}(N_*)$ is admissible.}
 
\begin{description}
	\item[predicate] $\in$.
	\item[constants] $\mathring{\sigma}, \mathring S, \underline x$ for $x \in L_{\delta_{N_*}}(N_*)$. 
	\item[axioms] \begin{itemize} \item $\ZFC^-$ and \textsf{Basic Axioms};
		\item $\mathring \sigma : \underline \N \prec \underline{N_*}$ is $\underline{\overline \kappa(0)}$-cofinal;
		\item $\mathring{\sigma}(\overline{\underline{\theta}}, \overline{\underline{\D}}, \overline{\underline{\U}}, \overline{\underline c})=\underline{\theta_*}, \underline{\D_*}, \underline{\U_*}, \underline{c_*}$;
		\item $\mathring S$ is a $\underline{\D_*}$-generic sequence over $\underline{N_*}$;
		\item $\mathring \sigma ``\overline{\underline S} \subseteq \mathring S$.
	\end{itemize}
\end{description} 
\end{definition}

As defined, we have that $\mathcal L_*$ is a $\Sigma_1(L_{\delta_{N_*}}(N_*))$-theory, since altogether the \textsf{Basic Axioms} are $\Sigma_1(L_{\delta_{N_*}}(N_*))$.

\begin{claimno} \label{claim:ConL0} $\mathcal L_0=\mathcal L(L_{\delta_{N_0}}(N_0))$ is consistent. \end{claimno}
\begin{proof}[Pf] Of course, it is not the case that $\sigma_0$ is $\overline \kappa(0)$-cofinal - all we know is that it is $\overline \nu$-cofinal. However, we know how to find an elementary embedding that is $\overline \kappa(0)$ cofinal: by taking a suitable transitive liftup.

Let $\langle N_1, \sigma_1 \rangle$ be the transitive liftup of $\langle \N, \sigma \rest H^{\N}_{\overline \kappa(0)} \rangle$, which exists by Interpolation (Fact \ref{fact:Interpolation}). So we have that $\sigma_1 \rest H^\N_{\overline \kappa(0)} = \sigma \rest H^\N_{\overline \kappa(0)}$. Let $\sigma_1(\overline \theta, \overline{\D}, \overline{\U}, \overline c) = \theta_1, \D_1, \U_1, c_1.$ 
Since $\sigma_1 \rest H^\N_{\overline \kappa(0)} = \sigma \rest H^\N_{\overline \kappa(0)} = \sigma_0 \rest H^\N_{\overline \kappa(0)}$, we also have a unique $k_1$ satisfying $ \text{$k_1: N_1 \prec N_0$ where $k_1\circ \sigma_1 = \sigma_0$ and $k_1 \rest \kappa_1(0) = \text{id}$ where $\kappa_1(0)=k_1(\kappa(0))$.}$ 

Figure \ref{figure:N1N0N} is a picture of all of the relevant transitive liftups.
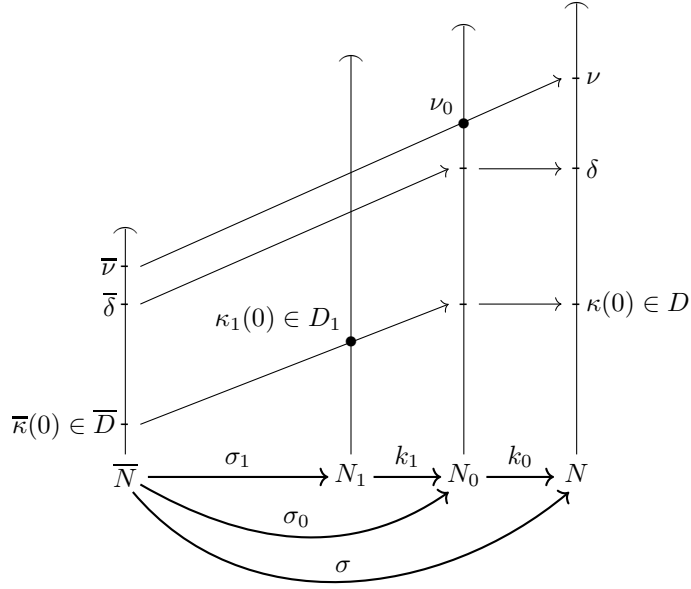
\begin{figure}[h!] 
\begin{tikzpicture}
\draw (-3,0)--(-3,3);
\node [below] at (-3,0) {$\N$};
\node at (-3,2.95) {$\frown$};

\node [left] at (-3,0.4) {$\overline \kappa(0) \in \overline D$};
\node at (-3,0.4) {-};
\node [left] at (-3,2) {$\overline \delta$};
\node at (-3,2) {-};
\node [left] at (-3,2.5) {$\overline \nu$};
\node at (-3,2.5) {-};

\draw [->] (-2.8,0.4)--(1.3,2);

\draw (0,0)--(0,5.3);
\node [below] at (0,0) {$N_1$};
\node at (0,5.25) {$\frown$};
\node at (0,1.5) {$\bullet$};
\node [above left] at (0,1.5) {$\kappa_1(0) \in D_1$};

\draw [->] (-2.8,2)--(1.3,3.8);

\draw (1.5,0) -- (1.5,5.7);
\node [below] at (1.5,0) {$N_0$};
\node at (1.5,5.65) {$\frown$};
\node at (1.5,2) {-};
\node at (1.5,3.8) {-};
\node [above left] at (1.5,4.4) {$\nu_0$};
\node at (1.5,4.4) {$\bullet$};

\draw [->] (1.7,2)--(2.8,2);
\draw [->] (1.7,3.8)--(2.8,3.8);

\draw [->] (-2.8,2.5)--(2.8,5);

\node [right] at (3,2) {$\kappa(0) \in D$};
\node at (3,2) {-};
\node [right] at (3,3.8) {$\delta$};
\node at (3,3.8) {-};
\node [right] at (3,5) {$\nu$};
\node at (3,5) {-};

\draw (3,0)--(3,6);
\node [below] at (3,0) {$N$};
\node at (3,5.95) {$\frown$};

\draw [->, thick] (-2.7,-0.3)--(-0.3,-0.3);
\draw [->, thick] (0.3, -0.3)--(1.2,-0.3);
\draw [->, thick] (1.8, -0.3)--(2.7,-0.3);
\draw[->, thick] (-2.8,-0.4) to  [out=-30, in=-145] node[above]{$\sigma_0$} (1.3,-0.5);
\draw[->,thick] (-2.9,-0.5) to [out = -50, in =-140] node[above]{$\sigma$} (2.9, -0.5);

\node [above] at (-1.5,-0.3) {$\sigma_1$};
\node [above] at (0.75,-0.3) {$k_1$};
\node [above] at (2.25, -0.3) {$k_0$};
\end{tikzpicture}\label{figure:N1N0N} \caption{Diagram depicting the liftups $\langle N_0, \sigma_0\rangle$ and $\langle N_1, \sigma_1\rangle$.}
\end{figure} 

We first show that $\mathcal L_1=\mathcal L(L_{\delta_{N_1}}(N_1))$ is consistent, by seeing that it has a model. To do this, we will find a sequence extending $\sigma_1``\S$ that is $\D_1$-generic over $N_1$. Then we will use the Transfer Lemma to see that this transfers to the consistency of $\mathcal L_0$. 
First, force over $V$ with $\D_1 = \D_1(\mathcal U_1)$, which is a generalized diagonal Prikry forcing over $\mathcal U_1$, to obtain a diagonal Prikry sequence $S_1'$. 
Define, in $V[S_1']$, a new sequence $S_1$ as follows:

$$S_1(\kappa) = \begin{cases} S_1'(\kappa) &\text{ if } \kappa \in D_1 \setminus \sigma_1``\overline D; \\
					\sigma_1(\overline S(\overline \kappa)) &\text{ if } \kappa = \sigma_1(\overline \kappa) \in \sigma_1``\overline D. \end{cases}$$

\begin{claimno} \label{Claim:S_1isPrikryoverN_1} The sequence $S_1$ is a $\D_1$-generic sequence over $N_1$. \end{claimno}
\begin{proof}[Pf.]
We will show that $S_1$ satisfies the generalized diagonal Prikry genericity criterion (Fact \ref{fact:diagprikrymathias}) over $N_1$. To do this, let $\mathcal X = \seq{ X_\kappa \in U_1(\kappa) }{ \kappa \in D_1 } \in N_1$, be a sequence of measure-one sets in the sequence of measures $\U_1$.

Note first that as $S_1'$ is a generic sequence, it already satisfies the generalized diagonal Prikry genericity criterion, namely:
$\set{ \kappa \in D_1}{ S_1'(\kappa) \notin X_\kappa } \text{ is finite.}$
Recall that $\S = \seq{ \S(\overline \kappa) }{ \overline \kappa \in \overline D }$ is a $\overline{\D}$-generic sequence as well.
We need to see that in addition,
$\set{ \overline \kappa \in \overline D}{\sigma_1(\S(\overline \kappa)) \notin X_{\sigma_1(\overline \kappa)} }  \text{ is finite,}$
since then 
\begin{align*}
\set{ \kappa \in D_1}{ S_1(\kappa) \notin X_\kappa } = & \set{ \kappa \in D_1 \setminus \sigma_1``\overline D}{ S_1'(\kappa) \notin X_\kappa } \cup\\
& \set{ \kappa = \sigma_1(\overline \kappa) \in \sigma_1``\overline D}{\sigma_1(\S(\overline \kappa)) \notin X_\kappa }
\end{align*}
is finite as well, completing the proof as desired.

By the $\overline{\kappa}(0)$-cofinality of $\sigma_1$, there is some $w \in \N$ such that $\mathcal X \in \sigma_1(w)$, where $|w| < \overline{\kappa}(0)$ in $\N$. Thus in $N_1$, we have that $|\sigma_1(w)| < \kappa_1(0)$. 
We may assume that $w$ consists of functions $f \in \prod_{\overline \kappa \in \overline D} \overline U(\overline \kappa)$.
So for each $\kappa \in \sigma_1``\overline D$, we have that $X_\kappa \in \sigma_1(w)_\kappa = \set{\sigma_1(f)(\kappa) }{ f \in \prod_{i \in \overline D} \overline U(i) \ \land \ f \in w }$ and also $|\sigma_1(w)_\kappa|<\kappa_1(0).$ So all $\kappa \in \sigma_1``\overline D$ of course satisfy $\kappa \geq \kappa_1(0)$ and thus by the $\kappa$-completeness of $U_1(\kappa)$, we have that $W_\kappa := \bigcap \sigma_1(w)_\kappa \in  U_1(\kappa).$
So we have established that $\mathcal W$, the sequence of $W_\kappa$ for $\kappa \geq \kappa_1(0)$, is also a sequence of measure-one sets in $N_1$. Note in addition that for $\kappa \in \sigma_1``\overline D$, we have that $W_\kappa \subseteq X_\kappa$. 

For each $\overline \kappa \in \overline D$, we have that $\overline W_{\overline \kappa} = \bigcap \set{f(\overline \kappa) }{ f \in \prod_{i \in \overline D} \overline U(i) \ \land \ f \in w }$ is a measure-one set in $\overline U(\overline \kappa)$ and we also have that $\sigma_1(\overline W_{\overline \kappa}) = W_{\sigma_1(\overline \kappa)}$, all by elementarity. Moreover, 
$\set{\overline \kappa \in \overline D}{\overline S(\overline \kappa) \notin \overline W_{\overline \kappa}} \text{ is finite}$ by the generalized diagonal Prikry genericity criterion applied to $\overline S$.
Thus by elementarity,
$$\set{ \overline \kappa \in \overline D }{ \sigma_1(\S(\overline \kappa)) \notin W_{\sigma_1(\overline \kappa)} } \supseteq \set{ \overline \kappa \in \overline D }{ \sigma_1(\S(\overline \kappa)) \notin X_{\sigma_1(\overline \kappa)}} \text{ is finite,}$$
as is desired, completing the proof of Claim \ref{Claim:S_1isPrikryoverN_1}.
\end{proof}

We have shown that $\mathcal L_1$ is consistent, by Barwise Correctness (Fact \ref{fact:correctness}), since we have just shown that $\langle H_{\delta}; \, \sigma_1, S_1 \rangle$ is a model of $\mathcal L_1$.

Let's check that we may now apply the Transfer Lemma (Fact \ref{fact:Transfer}) to the embedding $k_1:N_1 \prec N_0$. By Lemma \ref{lemma:liftupfull} we have that $N_1$ is almost full. We also have that $\mathcal L_1 = \mathcal L(L_{\delta_{N_1}}(N_1))$ is 
	$\text{$\Sigma_1(\langle N_1; \, \overline \theta, \overline{\D}, \overline{\U}, \overline c, \theta_1, \D_1, \U_1, c_1 \rangle)$}$ while $\mathcal L_0=\mathcal L(L_{\delta_{N_0}}(N_0))$ is 
	$\text{$\Sigma_1(\langle N_0;\,k_1(\overline \theta, \overline{\D}, \overline{\U}, \overline c, \theta_1, \D_1, \U_1, c_1)\rangle)$.}$
Furthermore $k_1$ is cofinal, since for each element $x \in N_0$, as $\sigma_0$ is cofinal, there is $u \in \N$ such that $x \in \sigma_0(u)$. Thus $\sigma_1(u) \in N_1$, and moreover $x \in k_1(\sigma_1(u))=\sigma_0(u)$. 

Therefore, we have that $\mathcal L_0$ is consistent, since $\mathcal L_1$ is consistent. This completes the proof of Claim \ref{claim:ConL0}. \end{proof}

From the consistency of $\mathcal L_0$ (recall that $\mathcal L_0$ is defined in Definition \ref{definition:L}), we would now like to use Barwise Completeness (Fact \ref{fact:completeness}) to obtain a model of $\mathcal L_0$. To do this, we need the admissible structure the theory is defined over to be countable. So let's work in $V[F]$, a generic extension that collapses $L_{\delta_{N_0}}(N_0)$ to be countable. Then by Barwise Completeness, $\mathcal L_0$ has a solid model 
	$\mathfrak A = \langle \mathfrak A; \mathring{S}^{\mathfrak A}, \mathring{\sigma}^{\mathfrak A} \rangle$ such that $\Ord \cap \wfc(\mathfrak A) = \Ord \cap L_{\delta_{N_0}}(N_0).$
We also have that $\mathring \sigma^{\mathfrak A}: \overline{\underline N}^{\mathfrak A} \prec \underline{N_0}^{\mathfrak A}$. By Fact \ref{fact:PointofBasicAxioms} we have that $\overline{\underline N}^{\mathfrak A}=\N$ and $N_0=\underline{N_0}^{\mathfrak A}$. Thus we may say that  $\mathring \sigma^{\mathfrak A}: \N \prec N_0$. 

Let $S = \mathring{S}^{\mathfrak A}$ and $\overset{*} {\sigma}=k_0 \circ \mathring{\sigma}^{\mathfrak A}$. 

Then $S$ is a $\D_0$-generic sequence over $N_0$, and as $k_0:N_0 \cong C$ we also have that $k_0``S$ is $\D$-generic over $C$. 

\begin{claimno} \label{claim:SisDgenericOverV} $S$ is $\D$-generic over $V$. \end{claimno}
\begin{proof}[Pf.]
We will again use the generalized diagonal Prikry genericity criterion, so let $\mathcal X = \seq{ X_\kappa \in U_1(\kappa) }{ \kappa \in D }$ be a sequence of measure-one sets in the sequence of measures $\U$. Fix a dense set $E \subseteq \D$ of size $\delta$ with $E \in C$. Since $\delta \subseteq C$, we have that $E \subseteq C$ as well. Find a condition $(s,A) \in E$ that strengthens $(\emptyset, \mathcal X)$. Thus for $\kappa \in \dom A$, we have that $A(\kappa) \subseteq X_\kappa$. Define a sequence of measure-one sets $B$ in $C$ so that 
$$B(\kappa) = \begin{cases} A(\kappa) &\text{if} \ \kappa \in \dom A; \\ \kappa &\text{if} \ \kappa \in \dom s. \end{cases}$$
So we have that $B$ is a sequence of measure-one sets in $C$. So $\set{ \kappa \in D }{S(\kappa) \notin B(\kappa) }$ is finite. Thus $\set{\kappa \in D}{ S(\kappa) \notin X_\kappa}$ is finite. So $S$ is $\D$-generic over $V$, completing the proof of Claim \ref{claim:SisDgenericOverV}.
\end{proof}

This will be the $\D$-generic sequence we need to satisfy the Main Claim. We will see in the following Claim \ref{claim:starsigmaisgood} that $\overset{*} \sigma$ has all of the desired properties, but it fails to be in $V[S]$. However, $\overset{*}{\sigma}$ will at least enable us to see that our $\in$-theory $\mathcal L'$ from Definition \ref{definition:T}, defined to assist us in proving the Main Claim, is consistent over a suitable admissible structure.
\begin{claimno} \label{claim:starsigmaisgood} The map $\overset{*}{\sigma}$ satisfies:
\begin{enumerate}
	\item \label{item:elemembed} $\overset{*} {\sigma}: \N \prec N$;
	\item \label{item:rangematches} $\overset{*} {\sigma}(\overline \theta, \overline{\D}, \overline{\U}, \overline c)=\theta, \D, \U, c$;
	\item \label{item:skolemC} $\sk{N}{\delta}{\overset{*}{\sigma}} = C$;	
	\item \label{item:lifts} $\overset{*} {\sigma} `` \S \subseteq S$.
\end{enumerate}
\end{claimno}
\begin{proof}[Pf.]
For (\ref{item:elemembed}), we have already seen above that $\mathring \sigma^{\mathfrak A}: \N \prec N_0$. Since $k_0:N_0 \prec N$, the desired result follows.

For (\ref{item:rangematches}), $\overset{*} {\sigma}(\overline \theta, \overline{\D}, \overline{\U}, \overline c)= k_0(\theta_0, \D_0, \U_0, c_0) =\theta, \D, \U, c$.

We know (\ref{item:skolemC}) holds since $N_0 = \sk{N_0}{\delta}{\mathring \sigma^{\mathfrak A}}$. To see this, clearly we have that $\sk{N_0}{\delta}{\mathring \sigma^{\mathfrak A}} \subseteq N_0$, since $\delta \in N_0$ as $N_0 \cong C$, and certainly $\ran(\mathring \sigma^{\mathfrak A})) \subseteq N_0$ as well. 
Then because $\mathring \sigma^{\mathfrak A}$ is $\overline \kappa(0)$-cofinal, by Lemma \ref{lemma:liftupchar}, we have, since $\mathring \sigma^{\mathfrak A}(\overline \kappa') < \delta$, that:
$$N_0 = \set{\mathring \sigma^{\mathfrak A}(f)(\beta)}{ f: \gamma \longrightarrow N_0, \ \gamma < \overline \kappa(0) \text{ and } \beta < \mathring \sigma^{\mathfrak A}(\gamma) } \subseteq \sk{N}{\delta}{\mathring \sigma^{\mathfrak A}}.$$
Thus $C = k_0``N_0=\sk{N_0}{\delta}{k_0 \circ \mathring \sigma^{\mathfrak A}}$ as desired. 

To see (\ref{item:lifts}), note that $\overset{*} {\sigma} \rest \overline \kappa(0) = \mathring{\sigma}^{\mathfrak A} \rest \overline \kappa(0)$ since $k_0 \rest \nu_0 = \id$.

This completes the proof of Claim \ref{claim:starsigmaisgood}.
\end{proof}

We are almost done, but recall that $\mathring{\sigma}^{\mathfrak A}$ is in $V[F]$, the collapse extension, and not in $V[S]$, and so $\overset{*} {\sigma}$ is also not necessarily in $V[S]$. We will use Barwise Completeness one last time, to finally find an embedding $\sigma'$ to satisfy the Main Claim along with the $S$ we found above.

Let $\lambda$ be regular in $V[S]$ with $N \in H_\lambda^{V[S]}$. Then 
	$$\M = \langle H^{V[S]}_\lambda;\ N, \sigma, S;\ \theta, \delta, \D, \U, c \rangle \text{ is admissible.}$$

In order to satisfy our Main Claim, we need a model of $\mathcal L'(\M)$ in $V[S]$. By Claim \ref{claim:starsigmaisgood}, we have that $\langle \M , \overset{*} {\sigma}\rangle$ is a model of $\mathcal L'(\M)$ in $V[F]$. Thus by Barwise Correctness (Fact \ref{fact:correctness}), $\mathcal L'(\M)$ is consistent.
Let 
	$\pi: \widetilde \M \prec \M$ where $\widetilde \M$ is countable and transitive.
Note that $\widetilde \M \in H_{\omega_1}^{V[S]}=H_{\omega_1}^V$ since diagonal Prikry forcing does not add bounded subsets to  $\kappa(0) > \omega_1$.\footnote{As Fuchs \cite[p.~966]{Fuchs:2005kx} points out, this result is a modification to the proof that generalized diagonal Prikry forcing preserves cardinalities.} 
We also have that $\overline{\underline N}^{\widetilde{\mathfrak A}} = \N$, since $\M$ sees that $\N$ is countable so $\widetilde{\M}$ sees that $\pi^{-1}(\N)$ is, and it follows that $\pi^{-1}(\N)=\N$.
Moreover, $\mathcal L'(\widetilde \M)$ is consistent in $V[F]$ and thus $V$, since any inconsistency could be pushed up via $\pi$ to one in $\mathcal L'(\M)$, contradicting the consistency of this latter theory that we showed in $V[F]$.

So by Barwise Completeness (Fact \ref{fact:completeness}) we have $\mathcal L'(\widetilde \M)$ has a solid model $\widetilde{\mathfrak A} = \langle \widetilde{\mathfrak A}; \dot{\sigma}^{\widetilde{\mathfrak A}} \rangle$ such that $\Ord \cap \text{wfc}(\widetilde{\mathfrak A}) = \Ord \cap \widetilde \M.$

Letting $\widetilde \sigma' = \dot{\sigma}^{\widetilde{\mathfrak A}}$ and $\sigma'=\pi \circ \widetilde \sigma'$, the Main Claim is now satisfied with $\sigma'$ and our $\D$-generic sequence $S$.

Indeed let us verify each of the properties of $\sigma'$ required by the Main Claim. The verification of these properties shall use the agreement between $\widetilde{\mathfrak A}$ and $\widetilde \M$ on the special constants of $\widetilde \M$ and on the ordinals. The fact that $\pi$ fixes $\N$ will be greatly taken advantage of.

First we show (\ref{item:mainclaimsigma'elem}) of the Main Claim. Let's say that $\varphi[\sigma'(\overline a)]$ holds in $N$. So $\varphi[\pi(\widetilde \sigma'(\overline a))]^N$ holds in $\M$. Thus $\varphi[\widetilde \sigma'(\overline a)]^{\pi^{-1}(N)}$ holds in $\widetilde \M$, and thus also in $\widetilde{\mathfrak A}$. Indeed we know that $\underline{\N}^{\widetilde{\mathfrak A}} = \N$.
This means that $\varphi[\overline a]$ holds in $\N$, as desired.

To see (\ref{item:mainclaim2}), we have $\sigma'(\overline \theta, \overline{\D}, \overline{\U}, \overline c)= \pi(\underline{\theta}^{\widetilde{\mathfrak A}}, \underline{\D}^{\widetilde{\mathfrak A}}, \underline{\U}^{\widetilde{\mathfrak A}}, \underline{c}^{\widetilde{\mathfrak A}})=\theta, \D, \U, c$.

For item (\ref{item:MainClaimC=Sk}), let $\widetilde N = \underline N^{\widetilde{\mathfrak A}}$, $\widetilde \sigma = \underline \sigma^{\widetilde{\mathfrak A}}$ and $\widetilde \delta = \underline{\delta}^{\widetilde{\mathfrak A}}$. So $\pi(\widetilde \delta)= \delta$, $\pi(\widetilde N)=N$, and $\pi(\widetilde \sigma)=\sigma$. Because of the way the $\in$-theory $\mathcal L'$ was defined, we already have: 
\begin{equation} \tag{$\star$} \label{eqn:SkDotSigma=SkTildeSigma} \sk{\widetilde N}{\widetilde \delta}{\widetilde \sigma'} = \sk{\widetilde N}{\widetilde \delta}{\widetilde \sigma}. \end{equation} To see the inclusion from left to right, note that by (\ref{eqn:SkDotSigma=SkTildeSigma}) $\ran{\widetilde \sigma'} \subseteq \sk{\widetilde N}{\widetilde \delta}{\widetilde \sigma}$. Thus 
	$\ran(\pi \circ \widetilde \sigma') \subseteq \sk{\pi(\widetilde N)}{\pi(\widetilde \delta)}{\pi(\widetilde\sigma)}.$
Making the same observation for the other direction, it follows that
$\sk{N}{\delta}{\sigma'} = \sk{N}{\delta}{\sigma}$ as desired.

To see item (\ref{item:mainclaimlifts}), note that $\overline{\underline S}^{\widetilde{\mathfrak A}} = \overline S$ since $\overline S \subseteq \overline N$. So we have already by the definition of $\mathcal L'$  that $\dot{\sigma}^{\widetilde{\mathfrak A}}``\overline S \subseteq \underline S^{\widetilde{\mathfrak A}}$. Thus $\pi \circ \dot{\sigma}^{\widetilde{\mathfrak A}}``\overline S \subseteq \pi `` \underline S^{\widetilde{\mathfrak A}} \subseteq S$ as desired. 
This completes the proof of the Main Claim.
\end{proof}
We have satisfied the Main Claim, which finishes the proof that $\D$ is subcomplete.
\end{proof}

We now describe some minor modifications of the proof that give further results.

The first point is that the above proof also shows that generalized Prikry forcing that adds a countable sequence to each measurable cardinal is subcomplete. Before stating the corollary let's define the forcing. Again let $D$ be an infinite discrete set of measurable cardinals. Let $\U = \seq{ U(\kappa) }{ \kappa \in D }$ be a list of measures associated to $D$. 
Let $\D^*(\U) = \D^*$ be defined in the same way as $\D(\U)$ except the stem of a condition, $s$, in $\D^*(\U)$ is a function with domain in $[D]^{<\omega}$ taking each measurable cardinal $\kappa \in \dom(s)$ to a finite set of ordinals $s(\kappa) \subseteq \kappa$. 
The upper parts are defined for each $\kappa\in D$, and the extension relation is defined in the same way, now a condition is strengthened by end-extending the stems on each coordinate, as well as the whole sequence of stems, and shrinking the upper parts.

We may again form a $\D^*$-generic sequence $S = S_G$ for a generic $G \subseteq \D^*$, and we may write $S = \seq{ S(\kappa) }{ \kappa \in D }$ where $S(\kappa)$ is a countable sequence of ordinals less than $\kappa$. The genericity criterion for generic diagonal Prikry sequences is as that for $\D$, stated in Fact \ref{fact:diagprikrymathias} (or see~\cite[Thm.~1]{Fuchs:2005kx}), with the modification that $S$ is $\D^*$ generic if and only if for all $\mathcal X$, the set $\set{ \alpha }{ \exists \kappa \in D \ \alpha \in S(\kappa) \setminus X_\kappa }$ is finite.

\begin{corollary} 
Let $D$ be an infinite discrete set of measurable cardinals. Let $\U = \seq{ U(\kappa) }{ \kappa \in D }$ be a list of measures associated to $D$. Then $\D^*(\U)$ is subcomplete.
\end{corollary}
\begin{proof}[Proof Sketch.]
The modifications are mostly notational, and the main one that needs to be made is to adjust the proof of Claim \ref{Claim:S_1isPrikryoverN_1}. Here we have $\D^*_1$, the generalized diagonal Prikry forcing as computed in $N_1$, as well as $\overline{\D}^*$ of $\N$, 
and $S_1$, which we would like to show is a $\D_1^*$-generic sequence over $N_1$ in this case. $S_1$ is defined as $\sigma_1 ``\overline S$, using a diagonal Prikry sequence $S_1'$ to fill in the missing coordinates, where $S_1'$ is obtained by forcing with $\D_1$ over $V$.

We will show that $S_1$ satisfies the generalized diagonal Prikry genericity criterion over $N_1$ and follow the proof of Theorem \ref{theorem:main}. To do this, let $\mathcal X = \seq{ X_\kappa \in U_1(\kappa) }{ \kappa \in D_1 }$, with $\mathcal X \in N_1$, be a sequence of measure-one sets in the sequence of measures $\U_1$.

Note first that $S_1'$ is a generic sequence, so it already satisfies the generalized diagonal Prikry genericity criterion, namely:
$\set{ \alpha }{ \exists \kappa \in D_1 \ \alpha \in S_1'(\kappa) \setminus X_\kappa }$ is finite.
Recall that $\S = \seq{ \S(\overline \kappa) }{ \overline \kappa \in \overline D }$ is a $\overline{\D}$-generic sequence over $\N$ as well.
We need to see that in addition, $\set{ \alpha }{ \exists \overline \kappa \in \overline D \ \alpha \in \sigma_1(\S(\overline \kappa)) \setminus X_{\sigma_1(\overline \kappa)} }$ is finite.

By the $\overline{\kappa}(0)$-cofinality of $\sigma_1$, there is some $w \in \N$ such that $\mathcal X \in \sigma_1(w)$, where $|w| < \overline{\kappa}(0)$ in $\N$. Thus in $N_1$, $|\sigma_1(w)| < \kappa_1(0)$. 
For each $\kappa \in \sigma_1``\overline D$, we have that $X_\kappa \in \sigma_1(w)_\kappa = \set{\sigma_1(f)( \kappa) }{ f \in \prod_{i \in \overline D} \overline U(i) \ \land \ f \in w }$ and also $|\sigma_1(w)_\kappa|<\kappa_1(0).$ All $\kappa \in \sigma_1``\overline D$ of course satisfy $\kappa \geq \kappa_1(0)$ so by the $\kappa$-completeness of $U_1(\kappa)$, we have that $W_\kappa := \bigcap \sigma_1(w)_\kappa \in U_1(\kappa)$ since $X_\kappa \in \sigma_1(w)_\kappa$.
So we have established that $\mathcal W$, the sequence of $W_\kappa$ for $\kappa \geq \kappa_1(0)$, is also a sequence of measure-one sets in $N_1$. Note in addition that for $\kappa \in \sigma_1``\overline D$, we have that $W_\kappa \subseteq X_\kappa$. 

For each $\overline \kappa \in \overline D$, we have that $\overline W_{\overline \kappa} = \bigcap \set{f(\overline \kappa) }{ f \in \prod_{i\in \overline D} \overline U(i) \ \land \ f \in w }$ is a measure-one set in $\overline U(\overline \kappa)$ and also that $\sigma_1(\overline W_{\overline \kappa}) = W_{\sigma_1(\overline \kappa)}$, all by elementarity. Moreover, 
$\set{ \alpha }{ \exists \overline \kappa \in \D \ \alpha \in \overline S(\overline \kappa) \setminus \overline W_{\overline \kappa} } \text{ is finite}$ by the generalized diagonal Prikry genericity criterion for $\overline{\D}$, which is satisfied by $\overline S$.
Thus by elementarity,
$$\set{ \alpha }{ \exists \overline \kappa \in \D \ \sigma_1(\S(\overline \kappa)) \setminus W_{\sigma_1(\overline \kappa)} } \supseteq \set{ \alpha }{ \exists \overline \kappa \in \D \ \sigma_1(\S(\overline \kappa)) \setminus X_{\sigma_1(\overline \kappa)}} \text{ is finite,}$$
as is desired, completing the proof of our modification of Claim \ref{Claim:S_1isPrikryoverN_1}. 
\end{proof}

One might consider a mixed version of $\D$ and $\D^*$, a poset that adds a single point below some measurable cardinals, and a cofinal $\omega$-sequence below others. This forcing is clearly subcomplete as well.

Finally we show that generalized diagonal Prikry forcing is subcomplete above $\mu$ where $\mu$ is a regular cardinal less than the first limit point of the forcing's associated sequence of measurables.

\begin{corollary}\label{corollary:SCaboveMu} Let $D$ be an infinite discrete set of measurable cardinals. Let $\U = \seq{ U(\kappa) }{ \kappa \in D }$ be a list of measures associated to $D$. 

Furthermore, let $\mu < \lambda$ be a regular cardinal, where $\lambda = \sup_{n<\omega} \kappa_n$, the first limit point of $D$.
Then $\D=\D(\U)$ is subcomplete above $\mu$. \end{corollary}

\begin{proof}[Proof Sketch] The idea is to follow the same exact proof as in Theorem \ref{theorem:main}, except we achieve Figure \ref{figure:SCaboveMu}.

\begin{figure}[h!] 
\begin{tikzpicture}
\draw (-3,-1)--(-3,3);
\node [below] at (-3,-1) {$\N$};
\node at (-3,2.95) {$\frown$};

\node [left] at (-3,-0.4) {$\overline \mu$};
\node at (-3,-0.4) {-};
\node [left] at (-3,0) {$\overline D \ni\overline \kappa' $};
\node at (-3,0) {-};
\node [left] at (-3,0.6) {$\overline \lambda$};
\node at (-3,0.6) {-};
\node [left] at (-3,2) {$\overline \delta$};
\node at (-3,2) {-};
\node [left] at (-3,2.5) {$\overline \nu$};
\node at (-3,2.5) {-};

\draw [->] (-2.8,0)--(1.3,1.4);

\draw (0,-1)--(0,5.3);
\node [below] at (0,-1) {$N_1$};
\node at (0,5.25) {$\frown$};

\node at (0,0.95) {$\bullet$};
\node [above left] at (0,0.95) {$\kappa_1' \in D_1$};

\draw [->] (-2.8,2)--(1.3,3.8);

\draw (1.5,-1) -- (1.5,5.7);
\node [below] at (1.5,-1) {$N_0$};
\node at (1.5,5.65) {$\frown$};
\node at (1.5,1.4) {-};
\node at (1.5,3.8) {-};
\node [above left] at (1.5,4.4) {$\nu_0$};
\node at (1.5,4.4) {$\bullet$};

\draw [->] (1.7,1.4)--(2.8,1.4);
\draw [->] (1.7,3.8)--(2.8,3.8);
\draw [->] (-2.8,2.5)--(2.8,5);

\node [right] at (3,.9) {$\mu$};
\node at (3,.9) {-};
\node [right] at (3,1.4) {$\kappa' \in D$};
\node at (3,1.4) {-};
\node [right] at (3,2.2) {$\lambda$};
\node at (3,2.2) {-};
\node [right] at (3,3.8) {$\delta$};
\node at (3,3.8) {-};
\node [right] at (3,5) {$\nu$};
\node at (3,5) {-};

\draw (3,-1)--(3,6);
\node [below] at (3,-1) {$N$};
\node at (3,5.95) {$\frown$};

\draw [->, thick] (-2.7,-1.3)--(-0.3,-1.3);
\draw [->, thick] (0.3, -1.3)--(1.2,-1.3);
\draw [->, thick] (1.8, -1.3)--(2.7,-1.3);
\draw[->, thick] (-2.8,-1.4) to  [out=-30, in=-145] node[above]{$\sigma_0$} (1.3,-1.5);
\draw[->,thick] (-2.9,-1.5) to [out = -50, in =-140] node[above]{$\sigma$} (2.9, -1.5);

\node [above] at (-1.5,-1.3) {$\sigma_1$};
\node [above] at (0.75,-1.3) {$k_1$};
\node [above] at (2.25, -1.3) {$k_0$};
\end{tikzpicture}
\caption{Diagram depicting the liftups involved with Corollary \ref{corollary:SCaboveMu}.}\label{figure:SCaboveMu}
\end{figure}
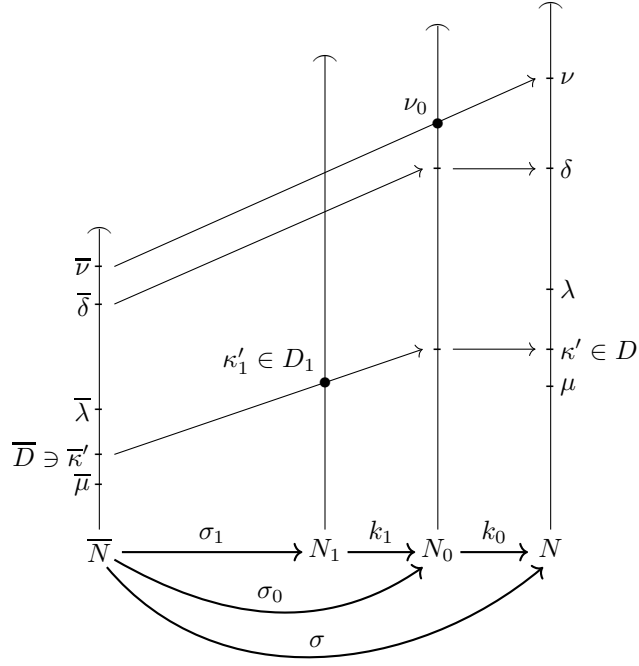 

Here we replace $\kappa(0)$ with some $\kappa' \in D$ such that $\mu < \kappa' $, where there are finitely many measurables of $D$ below $\kappa'$. 
So in particular, we let $\langle N_1, \sigma_1 \rangle$ be the liftup of $\langle \N, \sigma \rest H_{\overline{\kappa'}}^\N\rangle$ in Claim \ref{claim:N0isliftupofN}. In order to show the that we have a generic sequence over $\D_1$ as in Claim \ref{Claim:S_1isPrikryoverN_1}, we follow the same argument as follows:

Let $\mathcal X = \seq{ X_\kappa \in U_1(\kappa) }{ \kappa \in \sigma_1``\overline D }$, with $\mathcal X \in N_1$, be a sequence of measure one sets in the sequence of measures $\U_1$ with only coordinates coming from $\sigma_1``\overline D$.
We need to see that 
$\set{ \overline \kappa \in \overline D}{\sigma_1(\S(\overline \kappa)) \notin X_{\sigma_1(\overline \kappa)} }  \text{ is finite.}$
By the $\overline{\kappa}'$-cofinality of $\sigma_1$, there is some $w \in \N$ such that $\mathcal X \in \sigma_1(w)$, where $|w| < \overline{\kappa}'$ in $\N$. Thus in $N_1$, $|\sigma_1(w)| < \kappa_1'$. 
For each $\kappa \in \sigma_1``\overline D$, we have $$\textstyle X_\kappa \in \sigma_1(w)_\kappa = \set{\sigma_1(f)(\kappa) }{ f \in \prod_{i \in \overline D} \overline U(i) \ \land \ f \in w }$$ and also $|\sigma_1(w)|<\kappa_1'.$ So for all but finitely many $\kappa \in \sigma_1``\overline D$, namely for $\kappa \geq \kappa_1'$, by the $\kappa$-completeness of $U_1(\kappa)$, we have that $\bigcap \sigma_1(w)=W_\kappa \in U_1(\kappa).$

So we have established that $\mathcal W$, the sequence of $W_\kappa$ for $\kappa>\kappa_1'$, is also a sequence of measure-one sets in $N_1$. Note in addition that for $\kappa \in \sigma_1``\overline D$, $\kappa> \kappa_1'$, we have that $W_\kappa \subseteq X_\kappa$. 

By elementarity, for each $\overline \kappa \in \overline D$ with $\overline \kappa > \overline \kappa'$, $$\overline W_{\overline \kappa} =\cap \set{f(\overline \kappa) }{ f \in \textstyle\prod_{i \in \overline D} \overline U(i) \ \land \ f \in w }$$ is a measure-one set in $\overline U(\overline \kappa)$ and we also have that $\sigma_1(\overline W_{\overline \kappa}) = W_{\sigma_1(\overline \kappa)}$. Moreover, 
	$\set{\overline \kappa \in \overline D}{\overline S(\overline \kappa) \notin \overline W_{\overline \kappa}} \text{ is finite}$ by the generalized diagonal Prikry genericity criterion for $\overline{\D}$, which must be satisfied by $\overline S$, and since there are only finitely many measurables in $\overline D$ less than $\overline \kappa'$ in $\N$.
Thus by elementarity,
	$$\set{ \overline \kappa \in \overline D }{ \sigma_1(\overline S(\overline \kappa)) \notin W_{\sigma_1(\overline \kappa)} } \supseteq \set{ \overline \kappa \in \overline D }{ \sigma_1(\overline S(\overline \kappa) \notin X_{\sigma_1(\overline \kappa)}} \text{ is finite.}$$

Additionally the $\in$-theories $\mathcal L$ and $\mathcal L'$ would have to be defined so as to include as an axiom that $\mathring \sigma \rest \overline{\underline{\mu}} = \underline{\sigma} \rest \overline{\underline \mu}$ and $\dot \sigma \rest \overline{\underline{\mu}} = \underline{\sigma} \rest \overline{\underline \mu}$ respectively. 
We would then need to show that $\overset{*}\sigma \rest \overline{\mu} = \sigma \rest \overline{\mu}$, but this would follow since $k_0$ is the identity on $\nu_0$. 
Furthermore, it would need to be shown that $\sigma' \rest \overline{\mu} = \sigma \rest \overline{\mu}$, but this would follow from the requirement that $\widetilde \sigma' \rest \overline{\mu} = \underline{\sigma}^{\widetilde{\mathfrak A}} \rest \overline{\mu}$, and the fact that such ordinals are computed properly by $\widetilde{\mathfrak A}$. 
\end{proof}

\bibliographystyle{alpha}
\bibliography{../ApplicationStuff/CV-RS-Pubs-Refs-extra/BIB}
\end{document}